\documentclass[11pt]{amsart}
\usepackage{amsmath}

\usepackage{amsfonts,amssymb}
\usepackage{graphicx}
\usepackage{enumerate}
\usepackage{amsthm}
\usepackage[all]{xy}
\usepackage{extarrows}
\usepackage{color}
\usepackage{lmodern}
\usepackage{shuffle}
\usepackage{mathabx}
\usepackage{lmodern}
\usepackage{tikz}
\usepackage{hyperref}
\usepackage{rotating} 
\usepackage{scalefnt}

\usepackage{a4wide}

\newcommand {\bR}{\mathbb R}
\newcommand {\bN}{\mathbb N}
\newcommand {\bZ}{\mathbb Z}
\newcommand {\bC}{\mathbb C}

\newcommand {\bQ}{\mathbb Q}

\newcommand {\Id}{\operatorname{Id}}

\newcommand {\bk}{\mathbf{k}}
\newcommand {\bs}{\mathbf{s}}

\newcommand {\be}{\mathbf{1}}

\newcommand{\cA}{\mathcal{A}}
\newcommand{\cH}{\mathcal{H}}

\newcommand{\wt}{\operatorname{wt}}

\newcommand {\RE}{\operatorname{Re}}

\newcommand{\cS}{\mathcal S}

\newcommand{\fh}{\mathfrak{h}}

\newcommand{\oz}{\overline{\zeta}}

\newcommand {\cM}{\mathcal M}

\newcommand {\z}{\zeta}

\newtheorem{theorem}{Theorem}[section]
\newtheorem {lemma}[theorem]{Lemma}

\newtheorem {corollary}[theorem]{Corollary}

\newtheorem {example}[theorem]{Example}
\newtheorem {remark}[theorem]{Remark}

\makeatletter
\renewcommand\@biblabel[1]{#1.}
\makeatother

\begin{document}


\title[$q$-regularised multiple zeta values]{renormalisation of $q$-regularised multiple zeta values}


\author[K.~Ebrahimi-Fard]{Kurusch Ebrahimi-Fard}
\address{{ICMAT,
		C/Nicol\'as Cabrera, no.~13-15, 28049 Madrid, Spain}.
		{\tiny{On leave from UHA, Mulhouse, France.}}}
         \email{kurusch@icmat.es, kurusch.ebrahimi-fard@uha.fr}         
         \urladdr{www.icmat.es/kurusch}

\author[D.~Manchon]{Dominique Manchon}
\address{Univ. Blaise Pascal,
         C.N.R.S.-UMR 6620, 3 place Vasar\'ely, CS 60026, 63178 Aubi\`ere, France}       
      \email{manchon@math.univ-bpclermont.fr}
      \urladdr{http://math.univ-bpclermont.fr/~manchon/}

\author[J.~Singer]{Johannes Singer}
\address{Department Mathematik, 
		Friedrich-Alexander-Universit\"at Erlangen-N\"urnberg, 
		Cauerstra\ss e 11,  
		91058 Erlangen, Germany}
\email{singer@math.fau.de}
\urladdr{www.math.fau.de/singer}


\date{\today}

\maketitle


\begin{abstract}
We consider a particular one-parameter family of $q$-analogues of multiple zeta values. The intrinsic $q$-regularisation permits an extension of these $q$-multiple zeta values to negative integers. Renormalised multiple zeta values satisfying the quasi-shuffle product are obtained using an Hopf algebraic Birkhoff factorisation together with minimal subtraction. 
\end{abstract}


\section{Introduction}
\label{sec:intro}

For $\RE(s)>1$ the \emph{Riemann zeta function} is defined by 
\begin{align*}
 \zeta(s):=\sum_{n\geq 1} \frac{1}{n^s}, 
\end{align*}
which can be meromorphically continued to $\bC$ with a simple pole in $s=1$. It is well known that for even integers $k \in 2 \mathbb{N}$ we have
\begin{align*}
 \z(k)=  -\frac{(2\pi i)^k B_k}{2k!} 
\end{align*}
and for $k\in \bN_0$
\begin{align}
\label{eq:mero1}
 \z(-k) = -\frac{B_{k+1}}{k+1},  
\end{align}
where $B_k\in \bQ$ are the Bernoulli numbers defined by the generating series 
\begin{align}
\label{eq:Bernoulli}
 \frac{te^t}{e^t-1}= \sum_{k\geq 0} B_k\frac{t^k}{k!}. 
\end{align}

For $\bs:=(s_1,\ldots,s_n)\in \bC^n$ with $\sum_{i=1}^j {\RE(s_i)}>j$ ($j=1,\ldots,n$) \emph{multiple zeta values} (MZVs) are defined by the nested series 
\begin{align}\label{eq:MZV}
 \z(s_1,\ldots,s_n):=\sum_{m_1>\cdots >m_n>0} \frac{1}{m_1^{s_1} \cdots m_n^{s_n}}, 
\end{align}
which are natural generalisations of the Riemann zeta function (\cite{Krattenthaler07}). We call $n$ the \emph{depth} and $|\bs|:=s_1+\cdots+s_n$ the \emph{weight} of $\bs$. Usually MZVs are studied at positive integers $k_1,\ldots,k_n$ with $k_1\geq 2$. Especially the $\bQ$-vector space 
\begin{align*}
	\cM:=\langle \zeta(\bk)\colon \bk\in \bN^n, k_1\geq 2, n\in \bN_0 \rangle_{\bQ}
\end{align*}
is of great interest (\cite{Brown12,Zagier12}) since it is an algebra with two non-compatible products -- the \emph{shuffle} product and the \emph{quasi-shuffle} product. The latter is induced by the defining series \eqref{eq:MZV} applying the product rule of power series. The shuffle product is an application of the integration by parts formula for iterated Chen integrals using an appropriate integral representation of MZVs in terms of the 1-forms $\omega_j^{(1)}:=\frac{dt_j}{1-t_j}$ and $\omega_j^{(0)}:=\frac{dt_j}{t_j}$. Indeed, using the standard notation, $|\bk_{(j)}|:=k_1+\cdots+k_j$, $j\in\{1,\ldots,n\}$, one can show that \eqref{eq:MZV} can be written as
\begin{equation}
\label{eq:MZVint}
         \zeta(k_1,\ldots,k_n)=\int_{0}^{1} \bigg( \prod_{j=1}^{|\bk_{(1)}|-1}
                                \omega_j^{(0)}\bigg) \omega_{|\bk_{(1)}|}^{(1)} \cdots \bigg( \prod_{j=|\bk_{(n-1)}|+1}^{|\bk_{(n)}|-1}
                                \omega_j^{(0)}\bigg) \omega_{|\bk_{(n)}|}^{(1)} .
\end{equation}
This so-called double-shuffle structure gives rise to a great and intriguing variety of linear relations among MZVs (\cite{Hoffman97,Ihara06}), for instance the identity $\zeta(4)=4\zeta(3,1)$. 

The function $\z$ is meromorphic on $\bC^n$ with the subvariety $\cS_n$ of singularities, where 
\begin{align}\label{eq:meroMZV}
	\cS_n:=\left\{(s_1,\ldots,s_n)\in \bC^n\colon \begin{array}{l}
	s_1  = 1
	\text{~~or~~}
	s_1+s_2  = 2,1,0,-2,-4, \ldots 
	\text{~~or~~} \\  
	s_1+ \cdots +s_j  \in \bZ_{\leq j}~\text{for}~j=3,\ldots,n \end{array}  \right\}.
\end{align}
In \cite{Akiyama01a} it was shown that for $k_1,k_2\in \bN$ with $k_1+k_2$ odd we have 
\begin{align}\label{eq:mero2}
 \z(-k_1,-k_2) = \frac{1}{2}\frac{B_{k_1+k_2+1}}{k_1+k_2+1}.
\end{align}
In contrast to the dimension one case the meromorphic continuation of MZVs does not provide sufficient information for arbitrary negative integer arguments. This lack of data resulted in an interesting phenomenon related to the program commonly known as renormalisation of multiple zeta values at negative arguments. Using the fact that the quasi-shuffle product of MZVs can be abstracted to a quasi-shuffle Hopf algebra \cite{Hoffman00}, which is connected and graded, the authors in \cite{Guo08} and \cite{Manchon10} proposed renormalisation procedures that would allow the extension of any MZV to -- strictly -- negative arguments while preserving the quasi-shuffle product. Both approaches are based on the renormalisation Hopf algebra discovered by Connes and Kreimer \cite{Connes00,Connes01}, which permits to formulate the so-called BPHZ subtraction method in terms of an algebraic Birkhoff decomposition of regularised Hopf algebra characters. See \cite{Manchon08,Manchon10b,Panzer12} for introductions and more details. However, the two approaches, \cite{Guo08} and \cite{Manchon10}, propose different ways, both elaborated and sophisticated, of regularising MZVs at negative arguments. The resulting methods for constructing renormalised MZVs attribute different values to MZVs at negative arguments, while preserving the quasi-shuffle product. We remark that in \cite{Ebrahimi15} the authors presented a way to renormalise MZVs while preserving the shuffle product. Again, this approach is also based on the use of an algebraic Birkhoff decomposition of characters defined on a connected graded Hopf algebra, which is, however, different from the ones used in \cite{Guo08} and \cite{Manchon10}. It is an interesting question of how to relate the resulting zoo of different values for MZVs at negative arguments, based on renormalisation methods, which share the same algebraic workings in terms of an Hopf algebraic Birkhoff decomposition, but differ by applying distinct regularisation schemes.    

Through the works \cite{Castillo13a,Castillo13b} it became apparent that certain $q$-analogues of MZVs may exhibit double shuffle relations as well as an intrinsic regularisation that permits an extension of the resulting $q$-MZVs to negative arguments. A natural question that arises from this observation is whether such a $q$-regularisation can be used to renormalise MZVs while preserving -- one of -- the original shuffle-type products. In \cite{Ebrahimi15} this was shown to hold with respect to the shuffle product, for a certain $q$-analogue of MZVs known as the Ohno--Okuda--Zudilin (OOZ) model \cite{Ohno12}. In the light of the fact that there exist a variety of different $q$-analogues of MZVs, one may wonder whether a systematic approach is feasible, that characterises models of $q$-MZVs under the aspect of whether they provide a proper $q$-regularisation, which would permit a renormalisation of MZVs while preserving algebraic properties.             

The agenda of this note proposes a step in the aforementioned direction. Starting from a particular $q$-analogue of MZVs, its intrinsic $q$-regularisation permits its extension to negative arguments. The resulting $q$-MZVs at negative arguments are still represented as series of nested sums, which, moreover, exhibit a natural quasi-shuffle product. This fact allows us to present a rather transparent and simple construction of renormalised MZVs preserving the quasi-shuffle product. Moreover, our approach is enhanced by providing a one-parameter extension of this $q$-analogue, which results in infinitely many extensions rolled into a single approach. Regarding renormalisation of $q$-MZVs we remark that Zhao's approach in \cite{Zhao08} is rather different from the point of view just presented. It extended the approach of \cite{Guo08} to a particular $q$-analog of MZVs, without attributing any regularisation properties to the $q$-parameter itself.    

\smallskip

The aim of the note is as follows: 

\begin{itemize}
 \item We provide a one-parameter family of extensions of MZVs to negative integer arguments. Our approach is based on applying the renormalisation procedure of Connes and Kreimer combined with minimal subtraction to a particular $q$-regularisation of MZVs.  

 \item We show that the extended MZVs satisfy the quasi-shuffle product of MZVs, and coincide with the meromorphic continuation of MZVs whenever it is defined. 

 \item We exemplify that besides pure rational extensions of MZVs also irrational or complex values can appear as renormalised MZVs for certain deformation parameters, although the values given by the meromorphic continuation are always rational.
\end{itemize}
 
\smallskip

The note is organised as follows. In Section \ref{sec:qMZVs} we introduce a $q$-analogue of MZVs ($q$-MZVs) which is the starting point of this work. Section \ref{sec:quasi} is devoted to a review of the well known aspects concerning the quasi-shuffle Hopf algebra. The main results are stated in Section \ref{sec:main} and the proofs are given in Section \ref{sec:proofmain}. Finally, in Section \ref{sec:num} we provide some explicit numerical examples.

\vspace{0.4cm}

\noindent {\bf{Acknowledgements}}: The first author is supported by the Ram\'on y Cajal research grant RYC-2010-06995 from the Spanish government. He acknowledges support from the Spanish government under project MTM2013-46553-C3-2-P. The second author is supported by Agence Nationale de la Recherche (projet CARMA). We thank the referee for a subtle report, which helped improving the paper.


\section{$q$-analogues of MZVs}
\label{sec:qMZVs}

 In this section we introduce the $q$-analogue of MZVs which plays the central role in our construction. Let $q$ be a real number and $k_1,\ldots,k_n$ are positive integers. The Schlesinger--Zudilin model is defined by
\begin{align*}
 \z_q^{SZ}(k_1,\ldots,k_n):=\sum_{m_1>\cdots >m_n>0}\frac{q^{k_1 m_1+\cdots +  k_n m_n}}
 										{[m_1]_q^{k_1}\cdots [m_n]_q^{k_n}} \in \bQ[\![q]\!]
\end{align*}
with the $q$-number $[m]_q:=\frac{1-q^m}{1-q}$. 
For $0<q<1$ this series of nested sums is convergent \cite{Schlesinger01,Zudilin03,Singer15}. And if $k_1\geq 2$ we obtain MZVs defined in \eqref{eq:MZV} in the limit  $q\nearrow 1$. 

Our approach, however, is based on the following modification of the above model:
\begin{align}\label{eq:defSZ}
 \z_q(k_1,\ldots,k_n):=\sum_{m_1>\cdots >m_n>0}\frac{q^{|k_1| m_1+\cdots +  |k_n| m_n}}
 										{[m_1]_q^{k_1}\cdots [m_n]_q^{k_n}} \in \bQ[\![q]\!].
\end{align}    
As a result, this permits to consider the $q$-parameter as a natural regularisation, such that \eqref{eq:defSZ} may be defined for integer arguments $k_1,\ldots,k_n \in \bZ$.  In the following we refer to this as the \emph{regularised Schlesinger--Zudilin} model.
 
Let $D:=\{s\in \bC\colon \RE(s)>0\}$. For $t\in D$ we define a one-parameter family of \emph{modified} $q$-MZVs by
\begin{align}\label{eq:defSZmod}
 \oz_q^{(t),\ast}(k_1,\ldots,k_n):=\sum_{m_1>\cdots>m_n>0}{\frac{q^{(|k_1|m_1+\cdots+|k_n|m_n)t}}{(1-q^{m_1})^{k_1}\cdots (1-q^{m_n})^{k_n}}}. 
\end{align}
Since $0<q<1$ and $t\in D$, convergence of the previous series is always ensured for any $k_1,\ldots,k_n\in \bZ$ with $k_1\neq 0$. Indeed, since $0<q<1$ we observe that $(1-q)^k\leq (1-q^m)^k\leq 1$ for any $m\in \bN$ and $k\in \bN_0$. Furthermore $\left| q^{(|k_1|m_1+\cdots+|k_n|m_n)t} \right|\leq q^{|k_1|m_1 \RE(t)}$ because $k_1\neq 0$ and $t\in D$. Therefore 
\begin{align*}
\left|{\frac{q^{(|k_1|m_1+\cdots+|k_n|m_n)t}}{(1-q^{m_1})^{k_1}\cdots (1-q^{m_n})^{k_n}}}\right| \leq C q^{|k_1|m_1\RE(t)}, 
\end{align*}
where $C$ is a constant depending on $k_1,\ldots,k_n$. All in all we obtain 
\begin{align*}
 \left|\oz_q^{(t),\ast}(k_1,\ldots,k_n)\right| & \leq C \sum_{m_1>\cdots>m_n>0} q^{|k_1|m_1\RE(t)}
  = C \sum_{m_1,\ldots,m_n>0} q^{|k_1|(m_1+\cdots+m_n)\RE(t)}\\
 &  = C\left(\frac{q^{|k_1|\RE(t)}}{1-q^{|k_1|\RE(t)}} \right)^n <\infty.  
\end{align*}
See also \cite[Prop. 2.2]{Zhao07}.
Note that $ (1-q)^{|\bk|}\oz_q^{(1),\ast}(k_1,\ldots,k_n) =  \z_q(k_1,\ldots,k_n)$.


\section{The quasi-shuffle product}
\label{sec:quasi}

Recall that for integers $a,b>1$ multiplying the two sums 
\begin{equation}
\label{stuf}
	\sum_{m>0}\frac{1}{m^a}\sum_{n>0}\frac{1}{n^b}
	 = \sum_{m>n>0}\frac{1}{m^an^b} + \sum_{n>m>0}\frac{1}{n^bm^a} +  \sum_{m>0}\frac{1}{m^{a+b}},  
\end{equation}
which is known as Nielsen's reflexion formula $\zeta(a)\zeta(b)=\zeta(a,b)+\zeta(b,a)+\zeta(a+b)$. Defining the weight of products of MZVs as the sum of the weights of the factors, one notes the matching of weights on both sides. However, the length is not preserved. This generalises to arbitrary MZVs, and is commonly refereed to as quasi-shuffle product of MZVs. As an example we state the product 
$$
	\zeta(a,b)\zeta(c) = \zeta(c,a,b) + \zeta(a,b,c) + \zeta(a,c,b) + \zeta(a,b+c) + \zeta(a+c,b).
$$
for $a,c>1$ and $b>0$.

The key observation, which is not hard to verify, and which underlies our approach, is the fact that this simple product carries over to regularised Schlesinger--Zudilin $q$-MZVs, that is, the series of nested sums in \eqref{eq:defSZ} satisfy the quasi-shuffle product if all arguments are either strictly positive or strictly negative integers. For example the corresponding Nielsen reflexion formula for the regularised Schlesinger--Zudilin model defined in \eqref{eq:defSZ} for negative integers is
\begin{align*}
 \lefteqn{\z_q(-a)\z_q(-b) = \sum_{m>0}\left(q^{mt}[m]_q \right)^a\sum_{n>0}\left(q^{nt}[n]_q \right)^b}\\
 & = \sum_{m>n>0}\left(q^{mt}[m]_q \right)^a\left(q^{nt}[n]_q \right)^b+ \sum_{n>m>0}\left(q^{nt}[n]_q \right)^b\left(q^{mt}[m]_q \right)^a +\sum_{m>0}\left(q^{mt}[m]_q \right)^{a+b}\\
 & = \z_q(-a,-b)+\z_q(-b,-a)+\z_q(-a-b)
\end{align*}
for $a,b\in \bN$. One should note that the quasi-shuffle product is not preserved if we allow for integer arguments with mixed signs in \eqref{eq:defSZ}.

It is common to abstract the quasi-shuffle product using the words approach. Let $Y:=\{y_n\colon n\in \bZ\}$ be an alphabet. The set $Y^\ast$ of words defines a free monoid, with the empty word  denoted by $\be$. The linear span $\bQ\langle Y \rangle$ of words from $Y^*$ becomes a  commutative algebra when equipped with the \emph{quasi-shuffle product} $m_\ast\colon\bQ\langle Y \rangle \otimes \bQ\langle Y\rangle \to \bQ\langle Y \rangle$, $m_*(u \otimes v)=: u * v$ \cite{Hoffman12}. The latter is defined iteratively  for any words $u,v\in Y^\ast$ and letters $y_n,y_m \in Y$ by 
\begin{enumerate}[(i)] 
 \item $\be \ast u := u \ast \be := u$, 
 \item $y_n u \ast y_m v := y_n(u\ast y_m v) + y_m(y_n u\ast v) + y_{n+m}(u\ast v)$\hspace{0.5cm} ($n,m\in \bZ$),
\end{enumerate}
and extended to $\bQ\langle Y \rangle$ by distributivity. It is well-known that for $\fh^+:=\bQ\oplus \bigoplus_{n\geq 2}y_n\bQ\langle Y^+\rangle$  with $Y^+:=\{y_n\colon n\in \bN \}$ the map $\zeta^\ast\colon (\fh^+,m_\ast)\to (\bR,\cdot)$ defined by $\z^\ast(y_{k_1}\cdots y_{k_n}):=\z(k_1,\ldots,k_n)$ and $\z^\ast(\be):=1$ is an algebra morphism. 
For $\fh^-:=\bQ\langle Y^-\rangle$ with $Y^-:=\{y_n\colon n\in \bZ_{<0}\}$ we observe the following simple fact:

\begin{lemma}\label{lem:char}
The map $\oz_q^{(t)}\colon (\fh^-,m_\ast)\to (\bC[\![q]\!],\cdot)$ defined by $\oz_q^{(t)}(y_{k_1}\cdots y_{k_n}):=\oz_q^{(t),\ast}(k_1,\ldots,k_n)$ for any $k_1,\ldots,k_n\in \bZ_{<0}$ and $\oz_q^{(t)}(\be):=1$ is a morphism of algebras for any $t\in D$. 
\end{lemma}

Following \cite{Hoffman12} we introduce the deconcatenation coproduct $\Delta \colon \bQ\langle Y \rangle \to \bQ\langle Y \rangle \otimes \bQ\langle Y \rangle$ 
\begin{align}\label{eq:coproduct}
	\Delta(w) = \sum_{uv=w}u\otimes v
\end{align}
for $w \in Y^\ast$. Together with the quasi-shuffle product, $(\bQ\langle Y \rangle,m_\ast, \Delta)$ becomes a Hopf algebra. This construction directly applies to $Y^-$ and $\fh^-$ and therefore we observe: 

\begin{corollary}\label{cor:Hopf}
 The triple $(\fh^{-}, m_\ast,\Delta)$ is a graded, connected Hopf algebra, in which the grading is given by the weight $\wt(y_{k_1}\cdots y_{k_n}):= |k_1|+\cdots + |k_n|$. 
\end{corollary}

The reader is referred to \cite{Manchon08,Manchon10b} for details on connected graded Hopf algebras and related topics relevant to the forthcoming presentation.


\section{Renormalisation of MZVs}
\label{sec:main}

In order to extract renormalised MZVs from regularised expressions we use a theorem of Connes and Kreimer introduced in perturbative quantum field theory. See also \cite{Manchon08, Ebrahimi07a} for related results. Regularisation refers to a procedure of introducing a (family of) formal parameter(s) that renders an otherwise divergent expression formally finite. This step displays some freedom as how to regularise MZVs at negative arguments, and it may alter algebraic properties, which makes it therefore a nontrivial part of the renormalisation program.

\begin{theorem}[\cite{Connes00},\cite{Connes01}]
\label{theo:ConKre}
Let $(\cH,m_{\cH},\Delta)$ be a graded, connected Hopf algebra and  $\cA$ a commutative unital algebra equipped with a renormalisation scheme $\cA=\cA_{-}\oplus \cA_{+}$ and the corresponding idempotent Rota--Baxter operator $\pi$, where $\cA_{-}=\pi(\cA)$ and $\cA_{+}=(\Id-\pi)(\cA)$. Further let $\psi\colon \cH \to \cA$ be a Hopf algebra character, i.e., a multiplicative linear map from $\cH$ to $\cA$. Then the character $\psi$ admits a unique decomposition 
  \begin{align}\label{eq:birk}
   \psi=\psi_{-}^{\star{(-1)}}\star \psi_{+}
  \end{align}
called \emph{algebraic Birkhoff decomposition}, in which $\psi_{-}\colon \cH \to \bQ \oplus \cA_{-}$ and $\psi_{+}\colon \cH \to \cA_{+}$ are characters. The product on the right hand side of \eqref{eq:birk} is the convolution product defined on the vector space  $L(\cH, \cA)$ of linear maps from the $\cH$ to $\cA$.   
\end{theorem}

Recall that the vector space  $L(\cH,\cA)$ together with the convolution product $\phi \star \psi := m_{\cA} \circ (\phi \otimes \psi) \circ \Delta : \cH \to \cA$, where $\phi,\psi \in L(\cH,\cA)$, is an unital associative algebra. The set of characters is denoted by $G_\cA$ and forms a (pro-unipotent) group for the convolution product with (pro-nilpotent) Lie algebra $g_{\cA}$ of infinitesimal characters. The latter are linear maps $\xi \in L(\cH, \cA)$ such that for elements $x, y \in \cH$, both different from $\mathbf{1}$, $\xi(xy) = 0$. The exponential map $\exp^{\star}$ restricts to a bijection between $ g_{\cA}$ and $G_{\cA}$. The inverse of a character $\psi \in G_{\cA}$ is given by composition with the Hopf algebra antipode $S: \cH \to \cH$, e.g., $\psi_{-}^{\star{(-1)}} = \psi_- \circ S$. As a regularisation scheme we choose the commutative algebra $\cA:=\bQ[z^{-1},z]\!]$ with $\cA = \cA_{-}\oplus \cA_{+},$ where $\cA_{-}:=z^{-1}\bQ[z^{-1}]$ and $\cA_{+}:=\bQ[\![z]\!]$. On $\cA$ we define the corresponding projector $\pi: \cA \to \cA_{-}$ by 
\begin{align*}
 \pi\left(\sum_{n=-k}^\infty a_n z^n \right):= \sum_{n=-k}^{-1}a_n z^n
\end{align*}
with the common convention that the sum over the empty set is zero. Then $\pi$ and $\Id-\pi: \cA \to \cA_{+}$ are Rota--Baxter operators of weight $-1$ (see \cite{Connes00,Ebrahimi02,Ebrahimi07}). 

\smallskip

For the renormalisation of MZVs we perform two main steps. Firstly, we construct a regularised character $\psi\colon \fh^- \to \cA$ which means that we have to deform the divergent MZVs to a meromorphic function in the regularisation parameter $z$. Secondly, we apply an appropriate subtraction scheme. The latter is naturally given by Equation \eqref{eq:birk}, which essentially relies on the convolution product induced by the Hopf algebra $(\fh^{-},m_\ast,\Delta)$. The maps $\psi_+$ and $\psi_-$ of Equation \eqref{eq:birk} are recursively given by 
\begin{align}
 \psi_-(x) & =-\pi \left(\psi(x)+\sum_{(x)}\psi_{-}(x')\psi(x'') \right), \label{eq:psim} \\
 \psi_+(x)  & =(\Id-\pi) \left(\psi(x)+\sum_{(x)}\psi_{-}(x')\psi(x'') \right),\label{eq:psip}
\end{align}
for $x \in \cH$, $\wt(x) >0$, and $\psi_\pm \in G_\cA$. Note that we used Sweedler's notation for the reduced coproduct $\Delta'(x):=\sum_{(x)}x'\otimes x'':=\Delta(x)-\be \otimes x - x\otimes \be$. 
  
The map $\psi^{(t)}\colon (\fh^-,m_\ast) \to \bC[z^{-1},z]\!]$ defined by 
\begin{align}\label{eq:char}
 \psi^{(t)}(y_{-k_1}\cdots y_{-k_n})(z):= \frac{(-1)^{k_1+\cdots+k_n}}{z^{k_1+\cdots+ k_n}} \oz_{e^z}^{(t)}(y_{-k_1}\cdots y_{-k_n})
\end{align}
for $k_1,\ldots,k_n\in \bN$ is an algebra morphism (see Lemma \ref{lem:defexp} below). Applying Theorem \ref{theo:ConKre} we define \emph{renormalised MZVs} $\z_+^{(t)}$ by
\begin{align*}
 \z_+^{(t)}(-k_1,\ldots,-k_n):=\lim_{z\to 0} \psi_+^{(t)}(y_{-k_1}\cdots y_{-k_n})
\end{align*}
for any $k_1,\ldots,k_n\in \bN$. 

Now we present the main theorem of the paper which will be proven in the next section.
\begin{theorem}\label{theo:main}
Let $t\in D$. 
\begin{enumerate}[a)]
 \item The renormalised MZVs $\z_+^{(t)}$ verify the meromorphic continuation of MZVs , i.e., for any $\bk\in (\bZ_{<0})^n\setminus \cS_n$ we have $\z_+^{(t)}(\bk)=\z(\bk)\in \bQ$. 
 \item The renormalised MZVs $\z_+^{(t)}$ satisfy the quasi-shuffle product. 
 \item The renormalised MZVs $\z_+^{(t)}$ are rational functions in $t$ over $\bQ$ without singularities in $D$.
\end{enumerate}
\end{theorem}

As a consequence we obtain: 

\begin{corollary}\label{coro:main}
 \begin{enumerate}[a)]
  \item Let $t\in D\cap\bQ$ then $\z_+^{(t)}(\bk)\in \bQ$ for any $\bk\in (\bZ_{<0})^n$, $n\in \bN$. 
  \item Let $t\in D\cap \bR$ be transcendental over $\bQ$. Then there exists a $\bk\in (\bZ_{<0})^n$ such that $\z_+^{(t)}(\bk)\in \bR\setminus \bQ$.
 \end{enumerate}
\end{corollary}
\begin{proof}
 The first claim follows from Theorem \ref{theo:main} c). For the proof of b) see Section \ref{sec:num}. 
\end{proof}


\section{Proof of Theorem \ref{theo:main}}
\label{sec:proofmain}

In this section we implicitly use the following corollary of the fundamental theorem of algebra.
Let $p(t)\in \bC[t]$ be a polynomial over $\bC$ with $p(n)=0$ for countably infinitely many $n\in \bN$. Then $p=0$. 
Therefore some proofs in the subsequent paragraph -- concerning polynomial identities in $t$ -- are provided for $t\in \bN$ which then can be extended to $t\in D$ using this argument.    

\begin{lemma}\label{lem:defexp}
  Let $k_1,\ldots,k_n\in \bN$. Explicitly we have  
 \begin{align*}
  \psi^{(t)}(y_{-k_1}\cdots y_{-k_n})(z) = \sum_{m_1,\ldots,m_n\geq 0}\frac{B_{m_1}}{m_1!} \cdots \frac{B_{m_n}}{m_n!} C^{k_1,\ldots,k_n}_{m_1,\ldots,m_n}(t) z^{m_1+\cdots+m_n-(k_1+\cdots+k_n)-n}
 \end{align*}
 with 
 \begin{align}\label{eq:contC}
  C_{m_1,\ldots,m_n}^{k_1,\ldots,k_n}(t):=\sum_{l_1=0}^{k_1} \cdots \sum_{l_n=0}^{k_n}
  \prod_{j=1}^n\binom{k_j}{l_j} (-1)^{l_j+k_j+1}(l_1+k_1t+\cdots + l_j+k_jt)^{m_j-1}.
 \end{align}
 Furthermore $\psi^{(t)}\colon (\fh^-,m_\ast) \to \bC[z^{-1},z]\!]$ is an algebra morphism.
\end{lemma}

\begin{proof}
In Equation \eqref{eq:char} we defined $\psi^{(t)}$ by the following composition of maps
\begin{align*}
 \begin{tabular}{ccccccc}
   $(\fh^-,m_\ast)$ & $\longrightarrow$ &  $(\bC[\![q]\!],\cdot)$  &$\longrightarrow$ & $(\bC[z^{-1},z]\!],\cdot)$ & $\longrightarrow$ & $(\bC[z^{-1},z]\!],\cdot)$   \\
   $y_{k_1}\cdots y_{k_n}$ & $\longmapsto$ &  $\oz_q^{(t)}(y_{k_1} \cdots y_{k_n})$  &$\longmapsto$ & $\oz_{e^z}^{(t)}(y_{k_1}\cdots y_{k_n})$ & $\longmapsto$ & $(-z)^{|\bk|}\oz_{e^z}^{(t)}(y_{k_1} \cdots y_{k_n})$, 
\end{tabular}
\end{align*}
where $k_1,\ldots,k_n\in \bZ_{<0}$. By Lemma \ref{lem:char} the first map is an algebra morphism and the substitution map $q\mapsto e^z$ preserves this property. Since the quasi-shuffle product $m_\ast$ preserves the weight, the multiplication map is an algebra morphism, too. Hence, $\psi^{(t)}$ is a character. Performing the composition of maps we explicitly obtain for $k_1,\ldots,k_n \in \bN$
\allowdisplaybreaks{
\begin{align*}
   \lefteqn{\oz_q^{(t)}(y_{-k_1} \cdots y_{-k_n}) 
 =  ~\sum_{m_1>\cdots>m_n>0} q^{k_1 m_1 t}(1-q^{m_1})^{k_1} \cdots q^{k_n m_n t}(1-q^{m_n})^{k_n}} \\
 = & ~\sum_{m_1>\cdots>m_n>0}\prod_{j=1}^n q^{k_jm_jt}\sum_{l_j=0}^{k_j} \binom{k_j}{l_j}(-1)^{l_j}q^{m_jl_j}\\
 = & ~\sum_{m_1,\ldots,m_n>0} \sum_{l_1=0}^{k_1} \cdots \sum_{l_n=0}^{k_n}\prod_{j=1}^n \binom{k_j}{l_j} (-1)^{l_j}q^{m_j(l_1+k_1t+\cdots+l_j+k_jt)}\\
 = & \sum_{l_1=0}^{k_1} \cdots \sum_{l_n=0}^{k_n}\prod_{j=1}^n \binom{k_j}{l_j} (-1)^{l_j+1} \frac{q^{l_1+k_1t+\cdots+l_j+k_jt}}{q^{l_1+k_1t+\cdots+l_j+k_jt}-1}.
\end{align*}} 
 Now we rewrite the last expression by substituting $e^z$ for $q$ such that 
\allowdisplaybreaks{
\begin{align*}  
 \lefteqn{\oz_q^{(t)}(y_{-k_1} \cdots y_{-k_n})= \sum_{l_1=0}^{k_1} \cdots \sum_{l_n=0}^{k_n}\prod_{j=1}^n \binom{k_j}{l_j} (-1)^{l_j+1} \frac{q^{l_1+k_1t+\cdots+l_j+k_jt}}{q^{l_1+k_1t+\cdots+l_j+k_jt}-1}}\\ 
 \stackrel{q\mapsto e^z}{\mapsto} & \sum_{m_1,\ldots, m_n\geq 0} \frac{B_{m_1}}{m_1!}\cdots \frac{B_{m_n}}{m_n!}\sum_{l_1=0}^{k_1} \cdots \sum_{l_n=0}^{k_n}\prod_{j=1}^n \binom{k_j}{l_j} (-1)^{l_j+1} (z(l_1+k_1t+\cdots+l_j+k_jt))^{m_j-1},
\end{align*}}
where we used \eqref{eq:Bernoulli}. Multiplying the last expression by $(-z)^{-|\bk|}$ yields
\allowdisplaybreaks{
\begin{align*}  
(-z)^{-|\bk|}\oz_{e^z}^{(t)}(y_{-k_1} \cdots y_{-k_n}) 
=
 \sum_{m_1,\ldots, m_n\geq 0} \frac{B_{m_1}}{m_1!}\cdots \frac{B_{m_n}}{m_n!}C^{k_1,\ldots,k_n}_{m_1,\ldots,m_n}(t) z^{m_1+\cdots+m_n-(k_1+\cdots+k_n)-n}, 
\end{align*}}
which concludes the proof. 
\end{proof}

\begin{remark}{\rm{
 The particular choice of the map $(-\log(q))^{|\bk|}\oz_q^{(t),\ast}(\bk)$ in the definition of the character $\psi^{(t)}$ is necessary in order to assure that the renormalised MZVs verify the meromorphic continuation. Indeed, for the -- at first glance -- more natural looking choice  $(1-q)^{|\bk|}\oz_q^{(t),\ast}(\bk)$ in the definition of the character $\psi^{(t)}$ we observe a contradiction to the meromorphic continuation, e.g., we would obtain $\zeta_+^{(t)}(-1)=-\frac{1}{12}\frac{t^2+t-1}{t^2+t}\neq -\frac{1}{12}$ for any $t\in D$. 
}}\end{remark}

\begin{lemma}\label{lem:meroconst}
 Let $k\in \bN$. Then 
 \allowdisplaybreaks{
 \begin{align*}
  C^k_m(t) = \begin{cases}
           \sum_{l=0}^k\binom{k}{l}\frac{(-1)^{l+k+1}}{l+kt} & m=0, \\
           ~0 & m=1,\ldots, k,\\
           -k! & m=k+1.
          \end{cases}
 \end{align*}}
\end{lemma}

\begin{proof}
 In the case $m=0$ the claim is exactly Equation \eqref{eq:contC}. For $m\geq 1$ we observe
 \allowdisplaybreaks{
 \begin{align*}
  C_m^k(t) & =\left. (-1)^{k+kt+1} \sum_{l=0}^k \binom{k}{l} x^{l+kt}(l+kt)^{m-1}\right|_{x=-1} \\
  & = (-1)^{k+kt+1} \left.\left( x\partial_x \right)^{m-1} ((1+x)^kx^{kt}) \right|_{x=-1}\\
  & = \begin{cases}
        0 & m=1,\ldots, k,\\
       -k! & m=k+1, 
      \end{cases}
 \end{align*}}
which concludes the proof.
\end{proof}

\begin{lemma}\label{lem:mero1}
 For $k\in \bN$ we have
 \begin{align*}
  \psi^{(t)}_-(y_{-k})(z) = -C_0^k(t) z^{-k-1} \hspace{0.5cm} \text{~and~} \hspace{0.5cm} \psi^{(t)}_+(y_{-k})(z)  = -\frac{B_{k+1}}{k+1} + O_t(z),
 \end{align*}
where $O_t(z)$ is the standard Landau notation in which the subindex $t$ denotes the $t$-dependence of the coefficients of higher order terms.
\end{lemma}

\begin{proof}
 Since $y_{-k}$ is primitive in the Hopf algebra $(\fh^-,m_\ast,\Delta)$, Equations \eqref{eq:psim} and \eqref{eq:psip} imply respectively  
  \begin{align*}
   \psi^{(t)}_-(y_{-k})(z) = -\pi \psi^{(t)}(y_{-k})(z) \hspace{0.5cm}\text{~and~}\hspace{0.5cm} \psi^{(t)}_+(y_{-k})(z) = (\Id-\pi)\psi^{(t)}(y_{-k})(z).
  \end{align*}
From Lemma \ref{lem:meroconst} we deduce
\begin{align*}
 \psi^{(t)}_-(y_{-k})(z) &= -\pi\left( \sum_{m\geq 0}\frac{B_m}{m!}C^k_m(t) z^{m-k-1} \right) = -\sum_{m= 0}^k\frac{B_m}{m!}C^k_m(t) z^{m-k-1} = -C_0^k(t) z^{-k-1}
\end{align*}
and 
\begin{align*}
 \psi^{(t)}_+(y_{-k})(z) &= (\Id-\pi)\left(\sum_{m\geq 0}\frac{B_m}{m!}C^k_m(t) z^{m-k-1} \right) = \sum_{m\geq k+1}\frac{B_m}{m!}C^k_m(t) z^{m-k-1} 
  =- \frac{B_{k+1}}{k+1} + O_t(z). 
\end{align*}
\end{proof}

\begin{lemma}\label{lem:mero2}
 For $k_1,k_2 \in \bN$ with $k_1+k_2$ odd we have
 \begin{align*}
  \psi_+^{(t)}(y_{-k_1}y_{-k_2})(z) = \frac{1}{2} \frac{B_{k_1 + k_2 + 1}}{k_1 + k_2 + 1} + O_t(z).
 \end{align*}
\end{lemma}

\begin{proof}
First we remark that 
 \begin{align*}
  \Delta(y_{-k_1}y_{-k_2}) = \be \otimes y_{-k_1}y_{-k_2} +y_{-k_1}\otimes y_{-k_2} + y_{-k_1}y_{-k_2} \otimes \be 
 \end{align*}
 and therefore Equation \eqref{eq:psip} implies
 \begin{align*}
  \psi^{(t)}_+(y_{-k_1}y_{-k_2}) = (\Id-\pi)\left( \psi^{(t)}(y_{-k_1}y_{-k_2})+ \psi^{(t)}_-(y_{-k_1})\psi^{(t)}(y_{-k_2}) \right).
 \end{align*}
Since $k_1+k_2+2$ is odd and $B_s=0$ for odd $s\geq 2$ we obtain with $B_1=\frac{1}{2}$
 \begin{align*}
   (\Id-\pi)\psi^{(t)}(y_{-k_1}y_{-k_2})(z) &=  \sum_{m_1+m_2=k_1+k_2+2} \frac{B_{m_1}}{m_1!}\frac{B_{m_2}}{m_2!} C_{m_1,m_1}^{k_1,k_2}(t) + O_t(z) \\
   & = \frac{1}{2}\frac{B_{k_1+k_2+1}}{(k_1+k_2+1)!} \left( C^{k_1,k_2}_{k_1+k_2+1,1}(t)+C^{k_1,k_2}_{1,k_1+k_2+1}(t) \right) +O_t(z). 
 \end{align*}
On the one hand $\sum_{l_2=0}^{k_2}\binom{k_2}{l_2}(-1)^{l_2}=0$ and therefore  
\begin{align*}
 C_{k_1+k_2+1,1}^{k_1+k_2}(t)=\sum_{l_1=0}^{k_1} \sum_{l_2=0}^{k_2}\binom{k_1}{l_1} \binom{k_2}{l_2}(-1)^{l_1+l_2+k_1+k_2}(l_1+k_1t)^{k_1+k_2}=0. 
\end{align*}
On the other hand we get
\allowdisplaybreaks{
\begin{align*}
   & C_{1,k_1+k_2+1}^{k_1+k_2}(t)\\
 = &~ \left.(-1)^{k_1+k_2+(k_1+k_2)t} \sum_{l_1=0}^{k_1} \sum_{l_2=0}^{k_2} \binom{k_1}{l_1} \binom{k_2}{l_2} x^{l_1+l_2+(k_1+k_2)t} (l_1+l_2+(k_1+k_2)t)^{k_1+k_2}\right|_{x=-1} \\
 = &~ (-1)^{k_1+k_2+(k_1+k_2)t} \left.\left( x\partial_x\right)^{k_1+k_2}\sum_{l_1=0}^{k_1} \sum_{l_2=0}^{k_2} \binom{k_1}{l_1} \binom{k_2}{l_2} x^{l_1+l_2+(k_1+k_2)t} \right|_{x=-1}\\
 = &~ (-1)^{k_1+k_2+(k_1+k_2)t} \left. \left( x \partial_x\right)^{k_1+k_2}\left((1+x)^{k_1+k_2}x^{(k_1+k_2)t} \right)\right|_{x=-1}\\
 =&~ (k_1+k_2)!
\end{align*}}
and therefore 
\begin{align*}
 (\Id-\pi)\psi^{(t)}(y_{-k_1}y_{-k_2})(z) = \frac{1}{2}\frac{B_{k_1+k_2+1}}{k_1+k_2+1} +O_t(z).
\end{align*}
Moreover, since $B_{k_1+k_2+2}=0$ we observe with Lemma \ref{lem:defexp} and \ref{lem:mero1}
\begin{align*}
   (\Id-\pi)\left( \psi_-^{(t)}(y_{-k_1})(z)\psi^{(t)}(y_{-k_2})(z)\right)
   & = -(\Id-\pi)\left( C^{k_1}_0(t) z^{-k_1-1} \cdot \sum_{m\geq 0} \frac{B_m}{m!}C_m^{k_2}(t)z^{m-k_2-1}\right)\\
   &=O_t(z). 
\end{align*}

To summarise, we have shown that 
\begin{align*}
 \psi_+^{(t)}(y_{-k_1}y_{-k_2})(z) &=(\Id-\pi)\left( \psi^{(t)}(y_{-k_1}y_{-k_2})+ \psi^{(t)}_-(y_{-k_1})\psi^{(t)}(y_{-k_2}) \right)(z)\\
 &= \frac{1}{2}\frac{B_{k_1+k_2+1}}{k_1+k_2+1} + O_t(z), 
\end{align*}
which concludes the proof.
\end{proof}

\begin{proof}[Proof of Theorem \ref{theo:main}]
 From Lemma \ref{lem:mero1} we deduce for $k\in \bN$
 \begin{align*}
  \z_+^{(t)}(-k) = \lim_{z\to 0}\psi_+^{(t)}(y_{-k})(z) = -\frac{B_{k+1}}{k+1},
  \end{align*}
 which coincides with \eqref{eq:mero1} and for $k_1,k_2\in \bN$ with $k_1+k_2$ odd we observe from Lemma \ref{lem:mero2} that
   \begin{align*}
  \z_+^{(t)}(-k_1,-k_1) = \lim_{z\to 0}\psi_+^{(t)}(y_{-k_1}y_{-k_2})(z) = \frac{1}{2} \frac{B_{k_1+k_2+1}}{k_1+k_2+1},
  \end{align*}
which is consistent with \eqref{eq:mero2}. For length greater than three the meromorphic continuation does not provide any information for negative integer arguments since $(\bZ_{<0})^n\subseteq \cS_n$ for $n \geq 3$ (see \eqref{eq:meroMZV}). Therefore a) is established. Next we prove b). As shown in Lemma \ref{lem:defexp} $\psi^{(t)}$ is an algebra morphism with respect to the quasi-shuffle product $m_\ast$. Therefore Theorem \ref{theo:ConKre} implies that $\psi_+^{(t)}$ is also an algebra morphism and hence is $\z_+^{(t)}$. For c) we remark that Equation \eqref{eq:contC} of Lemma \ref{lem:defexp} shows that $\psi^{(t)}$ is a meromorphic function in $z$, whose coefficients are rational functions in $t$ over $\bQ$. Equations \eqref{eq:psim} and \eqref{eq:psip} indicate that the power series $\psi_+^{(t)}$ is obtained from $\psi^{(t)}$ by subtractions and projections. Therefore $\z_+^{(t)}$ is -- as the constant coefficient of the power series $\psi_+^{(t)}$ -- a rational function in $t$ over $\bQ$. 
\end{proof}


\section{Numerical examples}
\label{sec:num}

In this subsection we provide some explicit numerical examples of the renormalisation process introduced in the previous sections. First we provide an explicit example of how to compute renormalised MZVs. 
\begin{example}
 Let us calculate the renormalised MZV $\z_+^{(1)}(-1,-3)$. Using \eqref{eq:coproduct} we observe 
 \begin{align*}
  \Delta(y_{-1}y_{-3}) = \be \otimes y_{-1}y_{-3}+ y_{-1}\otimes y_{-3}+ y_{-1}y_{-3}\otimes \be.
 \end{align*}
Consequently, formula \eqref{eq:psip} implies 
\begin{align*}
 \psi_+^{(1)}(z)=(\Id-\pi) \left(\psi^{(1)}(y_{-1}y_{-3})(z)+\psi_{-}^{(1)}(y_{-1})(z)\psi^{(1)}(y_{-3})(z) \right).
\end{align*}
Using 
\begin{align*}
  \psi^{(1)}(y_{-1})(z)&= \frac{1}{2}{z}^{-2}-\frac{1}{12}+\frac{7}{720}{z}^{2}-
\frac{31}{30240}{z}^{4}+ O({z}^{5}), \\
  \psi^{(1)}(y_{-3})(z)& = \frac{1}{60}{z}^{-4}+\frac{1}{120}-\frac{41}{1008}{z}^{2}+
  \frac{2203}{28800}{z}^{4}+O({z}^{5}),\\
  \psi^{(1)}(y_{-1}y_{-3})(z) &=  \frac{3}{560}{z}^{-6}+\frac{1}{560}{z}^{-5}+\frac{47}{11200}{z}^{-2}-\frac{5377}{282240}+\frac{1}{84}z+O({z}^{2}), 
\end{align*}
we obtain 
\begin{align*}
 (\Id-\pi)\psi^{(1)}(y_{-1}y_{-3})(z) = -\frac{5377}{282240}+\frac{1}{84}z+O({z}^{2}). 
\end{align*}
Since $y_{-1}$ is a primitive element of the Hopf algebra $(\fh^{-}, m_\ast,\Delta)$ we have $\psi_-^{(1)}(y_{-1})(z)= -\frac{1}{2}{z}^{-2}$ by Equation \eqref{eq:psim} and therefore 
\begin{align*}
 (\Id-\pi) \left(\psi_{-}^{(1)}(y_{-1})(z)\psi^{(1)}(y_{-3})(z) \right) = \frac{41}{2016} +O(z^2).
\end{align*}
All in all we get  
\begin{align*}
  \psi_+^{(1)}(z)=(\Id-\pi) \left(\psi^{(1)}(y_{-1}y_{-3})(z)+\psi_{-}^{(1)}(y_{-1})(z)\psi^{(1)}(y_{-3})(z) \right) = \frac{121}{94080}+\frac{1}{84}z+O(z^2),
\end{align*}
which results in 
\begin{align*}
 \z_+^{(1)}(-1,-3)=\lim_{z\to 0} \psi_+^{(1)}(z) =  \frac{121}{94080}. 
\end{align*}
\end{example}
In Table \ref{table1} we list the renormalised MZVs in the case $t=1$ for depth two. For depth one the renormalised MZVs are always rational as well as for $k_1,k_2\in \bN$ with $k_1+k_2$ odd, due to Theorem \ref{theo:main} a). Because of the quasi-shuffle relation 
\begin{align*}
 \z_+^{(t)}(-k)\z_+^{(t)}(-k) = 2 \z_+^{(t)}(-k,-k) + \z_+^{(t)}(-2k)
\end{align*}
for $k\in \bN$ the diagonal entries $\z_+^{(t)}(-k,-k)$ are also always rational and do not depend on the parameter $t$. The first case for which we obtain a non-constant rational function in $t$ over $\bQ$ is the case $\z_+^{(t)}(-1,-3)$. The quasi-shuffle product implies
\begin{align*}
  \z_+^{(t)}(-1)\z_+^{(t)}(-3) =  \z_+^{(t)}(-1,-3)+\z_+^{(t)}(-3,-1) + \z_+^{(t)}(-4). 
\end{align*}
Therefore a priori $\z_+^{(t)}(-1,-3)$ and $\z_+^{(t)}(-3,-1)$ are not explicitly given by that relation. They only have to satisfy $ \z_+^{(t)}(-1,-3)+\z_+^{(t)}(-3,-1) = \z_+^{(t)}(-1)\z_+^{(t)}(-3)$. Using  \eqref{eq:psim} and \eqref{eq:psip} we find 
\allowdisplaybreaks{
  \begin{align*}
   \z_+^{(t)}(-1,-3) & = \frac{1}{8064}\frac{166t^2+166t+31}{(4t+3)(4t+1)}, \\
   \z_+^{(t)}(-3,-1) & = -\frac{1}{40320}\frac{1278t^2+1278t+239}{(4t+3)(4t+1)}. 
  \end{align*}}

This example also proves the second claim of Corollary \ref{coro:main}. Indeed, since not all renormalised MZVs $\z_+^{(t)}$ are constant as a rational function in $t$ over $\bQ$ the choice of a transcendental $t\in D$ leads to irrational values for $\z_+^{(t)}$. It is easily seen that also complex numbers can appear as renormalised MZVs.

{\tiny{
\begin{center}
\begin{table}[]
\renewcommand{\arraystretch}{2}
\begin{tabular}{r|c|c|c|c|c|c}
  {$k_1 \diagdown k_2$}& $-1$ & $-2$ & $-3$ & $-4$ & $-5$ & $-6$  \\
  \hline 
  $-1$ & $ \frac{1}{288}$ & $-\frac{1}{240} $ & $ \frac{121}{94080} $ & $\frac{1}{504}$ & $-\frac{31093}{17740800}$ & $-\frac{1}{480} $  \\
  \hline
  $-2$ & $-\frac{1}{240} $ & $0$ & $\frac{1}{504} $ & $-\frac{48529}{66528000} $ &$-\frac{1}{480}$ & $\frac{131679179}{71922090240}$  \\
  \hline
  $-3$ & $-\frac{559}{282240} $ & $\frac{1}{504} $ & $ \frac{1}{28800} $ & $ -\frac{1}{480} $ & $\frac{941347763}{1150753443840} $ & $\frac{1}{264}$ \\
  \hline
  $-4$ & $\frac{1}{504} $ & $\frac{48529}{66528000}  $ & $-\frac{1}{480} $ & $ 0$ &  $\frac{1}{264}$& $-\frac{199275989809861}{128121575662080000}$  \\
  \hline
  $-5$ & $ \frac{110879}{53222400} $ & $-\frac{1}{480}$ & $-\frac{979401779}{1150753443840}$ & $\frac{1}{264} $ & $\frac{1}{127008}$ & $-\frac{691}{65520}$  \\
  \hline
  $-6$ & $ -\frac{1}{480} $ & $-\frac{131679179}{71922090240}  $ & $\frac{1}{264}  $ & $\frac{199275989809861}{128121575662080000} $ & $-\frac{691}{65520}$ & $0$  \\
\end{tabular}
\caption{The renormalised MZVs $\zeta^{(1)}_+(k_1,k_2)$.}\label{table1}
 \end{table}
\end{center}}}


\subsection*{Conclusion}

Extensions of MZVs to strictly negative arguments by renormalisation, i.e., by using subtraction methods on properly regularised nested sums, implied by a Hopf algebraic Birkhoff decomposition, were introduced in \cite{Guo08} and \cite{Manchon10}. 
The two approaches differ in the regularisations used to obtain formally well-defined expressions. This results in different values associated to certain MZVs at negative arguments, while both approaches preserve the quasi-shuffle product. In this paper, we have shown how to renormalise MZVs using minimal subtraction in an intrinsically regularised $q$-analogue of MZVs. 
The subtraction method is implied by a Hopf algebraic Birkhoff decomposition, and the quasi-shuffle product is preserved. 

Regarding the parameter $t$ of $\z_+^{(t)}$, it is natural to ask for the relation between different sets of renormalised MZVs corresponding to different parameters. Let $t_1,t_2 \in D$. Recent work shows that we have $\z_+^{(t_1)} = \alpha_{t_1,t_2} \star \z_+^{(t_2)}$, where $\alpha_{t_1,t_2}$ belongs to the \emph{renormalisation group of MZVs} \cite{Ebrahimi15c}, which is a particular subgroup of the group of characters of the quasi-shuffle Hopf algebra. Similarly one can compare the renormalised MZVs in this work with those values obtained in the literature \cite{Guo08,Manchon10}. The reader is referred to \cite{Ebrahimi15c} for more details.

Moreover, in the light of the results  in \cite{Ebrahimi15} one may wonder whether the approach presented here can be further developed by characterising different $q$-analogues of MZVs as distinct $q$-regularisation methods for MZVs. A possible prescription would be to start by replacing the summation variables in the nested series \eqref{eq:MZV} by $q$-integers, and then to introduce in the nominator a particular polynomial in the $q$-regularisation parameter. The choice of the latter is crucial since it determines the algebraic as well as analytic properties of the resulting $q$-MZVs. Several $q$-analogues of MZVs have been described in the literature, and it would be rather interesting to understand them from the point of view of extending MZVs to negative arguments.

\bibliographystyle{plain}
\bibliography{library}

\end{document}